\documentclass[10pt,twoside,reqno]{amsart}
\usepackage[all]{xy}
\usepackage{amssymb, latexsym, amsthm, amsmath, mathtools, mathrsfs}
 \usepackage{enumitem,color}
 
\usepackage{multicol}
\usepackage{multirow}
\usepackage{mathtools}
\usepackage{soul}

\usepackage[numbers]{natbib} 
\usepackage{array,longtable}

\usepackage{booktabs}
\usepackage{bigstrut}
\usepackage{hyperref}
\usepackage[capitalize,nameinlink]{cleveref} 
\usepackage{xcolor} 
\hypersetup{
    colorlinks,
    linkcolor={black},
    citecolor={black},
    urlcolor={black}
}
\usepackage{soul}
\long\def\symbolfootnote[#1]#2{\begingroup%
\def\thefootnote{\fnsymbol{footnote}}\footnote[#1]{#2}\endgroup}

\newcommand{\A}{\mathscr{A}}
\newcommand{\Aut}{\textup{Aut}}

\newcommand{\Oc}{{\mathbf{O}}}
\newcommand{\bx}{\ensuremath{\mathbf{x}}}

\newcommand{\by}{\ensuremath{\mathbf{y}}}
\newcommand{\F}{\mathbb{F}}
\newcommand{\suchthat}{\;\ifnum\currentgrouptype=16 \middle\fi|\;}

\makeatletter
\def\imod#1{\allowbreak\mkern10mu({\operator@font mod}\,\,#1)}
\makeatother

\DeclareUnicodeCharacter{0306}{\v}
\newtheorem{mainthm}{Theorem}

\newtheorem{theorem}{Theorem}[section]

\newtheorem{lemma}[theorem]{Lemma}
\newtheorem{corollary}[theorem]{Corollary}
\newtheorem{proposition}[theorem]{Proposition}
\newtheorem*{theorem*}{Theorem}
\theoremstyle{definition}

\numberwithin{equation}{section}
\counterwithin{table}{section}

\newcommand{\ignore}[1]{}

\newcommand{\mynote}[1]{}
\begin{document}
\setcounter{section}{0}
\title[Polynomial maps with constants on split octonion algebras]{Polynomial maps with constants on \\ split octonion algebras}
\today
\author[Panja S.]{Saikat Panja}
\address{Indian Statistical Institute, Bangalore Center, Mysore road, Bangalore 560 059, India}
\email{panjasaikat300@gmail.com}
\author[Saini P.]{Prachi Saini}
\address{IISER Pune, Dr. Homi Bhabha Road, Pashan, Pune 411 008, India}
\email{prachi2608saini@gmail.com}
\author[Singh A.]{Anupam Singh}
\address{IISER Pune, Dr. Homi Bhabha Road, Pashan, Pune 411 008, India}
\email{anupamk18@gmail.com}
\thanks{The first-named author is supported by an NBHM postdoctoral fellowship, file number ending at R\&D-II/6746. The second-named author acknowledges the support of CSIR PhD scholarship number 09/0936(1237)/2021-EMR-I. The third-named author is funded by an NBHM research grant 02011/23/2023/NBHM(RP)/RDII/5955 for this research.}

\begin{abstract}
Let $\mathbf{O}(\mathbb{F})$ be the split octonion algebra over an algebraically closed field $\mathbb{F}$. For positive integers $k_1, k_2\geq 2$, we study surjectivity of the map $A_1(x^{k_1}) + A_2(y^{k_2}) \in \mathbf{O}(\mathbb{F})\langle x, y\rangle$ on $\mathbf{O}(\mathbb{F})$. For this, we use the orbit representatives of the ${G}_2(\mathbb{F})$-action on $\mathbf{O}(\mathbb{F}) \times \mathbf{O}(\mathbb{F}) $ for the tuple $(A_1, A_2)$, and characterize the ones which give a surjective map.
\end{abstract}
\dedicatory{To Professor B. Sury for his immense contribution to Algebra and Number Theory}
\maketitle
\section{Introduction}\label{sec:intro}
Lagrange's four-square theorem states that a non-negative integer can be written as a sum of four squares. Generalizing this is the \textit{Waring problem}, which asks whether every natural number can be expressed as a sum of finitely many powers of non-negative integers. More formally, for a fixed integer $ k \geq 2 $, the problem asks for the smallest number $g(k)$ such that every natural number $ n $ can be written as  
$$
n = X_1^k + X_2^k + \dots + X_{g(k)}^k
$$
for some integers $ X_1, X_2, \dots, X_{g(k)} $.  
This was initially conjectured by {Waring} in 1770, the existence was proved by {Hilbert} in 1909, and values of $g(k)$ for various $k$ are being determined.  

Several natural generalizations have been considered; here, we mention a few. Throughout the article, for a ring $R$, the algebra of $m\times m$ matrices with entries from $R$ will be denoted as $M_m(R)$. In 1985, Newman \cite{Newman1985} proved that every element of $M_2(\mathbb{Z})$ is a sum of $3$ squares. Vaserstein in 1986, \cite{Vaserstein1986} proved that every element in $M_n(\mathbb{Z})$ is a sum of three squares (also see \cite{Vaserstein1987}. The more general case of this is considered in a recent article \cite{koo2025integral}. Richman in 1987, \cite{Richman1987} proved that for a commutative associative ring $R$ with unity and positive integers $n, k\geq 2$, an element $A\in M_n(R)$ is a sum of $k$-th powers if and only if it is a sum of seven $k$-th powers. Katre and Khule, \cite{KatreKhule00} proved that if $R$ is an order in an algebraic number field $K$, then for $n\geq k\geq2$, every $n\times n$ matrix over $R$ is a sum of at most $7$ $k$-th powers if and only if $(k,\mathrm{disc}(R)) = 1$. Kishore and Singh in 2024 (see \cite{KishoreSingh2022}) proved that for a positive integer $k\geq 1$ and a finite field $\F_q$ of a sufficiently large order (depending solely on $k$), every element of $M_n(\F_q)$ is a sum of two squares (also see \cite{PanjaSainiSingh2023} for further generalisation). Waring problems for upper-triangular matrix algebra have been considered in \cite{KaushikSingh2023}.

The Waring problem has further generalizations: the study of polynomial maps on algebras. Let $\F$ be a field. For an $\F$-algebra $\A$ and an element $w\in \mathcal{F}_n := \F\langle X_1, \ldots, X_n\rangle$, the free algebra of rank $n$, we get a natural map
\begin{align*}
    \widetilde{w}\colon\A^n\longrightarrow \A,
\end{align*}
by substitution. These are known as \emph{polynomial maps} on $\A$. Kaplansky, in 1957, conjectured that the image of a multilinear map on the matrix algebra over an infinite field is a vector space (see \cite{Kaplansky1957}). This has attracted a lot of attention in recent years and is known as L'vov-Kaplansky conjecture. In \cite{BeMaRo12}, Kanel-Belov, Malev, and Rowen proved the conjecture for $M_2(K)$ where $K$ is a quadratically closed field. Br\v{e}sar in 2020, \cite{Bresar2020} proved that for an algebraically closed field $\F$ of characteristic $0$, a polynomial $f\in\mathcal{F}_m$ which is neither an identity nor a central polynomial of $M_n(\F)$, the trace-zero matrices can be written as sum of four elements from the set $f(A)-f(A)$. One can further ask whether four is the optimal choice. Answers to these questions have been described in \cite{BresarSemerl2021} and \cite{BresarSemerl2022}. Similar questions for upper-triangular matrix algebras have been addressed in \cite{PaPr23} and \cite{Chen2024}.

Let $\F$ be an algebraically closed field and $\Oc(\F)$ denote the octonion algebra over $\F$ which is non-commutative and non-associative. In this article, we concentrate on the map induced by the following polynomial of the free algebra $\Oc * \mathcal{F}_n$,
\begin{align*}
    A_1(X_1^{k_1}) + A_2(X_2^{k_2}) + \cdots + A_n(X_n^{k_n}),
\end{align*}
where $A_i\in \Oc(\F)$, and $k_i\geq 2$ are positive integers, for $1\leq i\leq n$. The maps given by elements of $\Oc * \mathcal{F}_n$ are called \textit{polynomial maps with constant}. In \cite{LopatinZubkov2024b}, Lopatin and Zubkov have considered the equations of the form $A(XB) = C$ and $A(BX) = C$ for $A, B, C\in \Oc(\F)$, and in \cite{lopatin2024polynomial} the polynomial equations of $1$ variable are considered. The images of multilinear and semihomogeneous polynomials on the algebra of octonions have been considered in \cite{BelovMalevPinesRowen2024} by Kanel-Belov, Malev, Pines, and Rowen.

In an earlier work \cite[Theorem A]{PanjaSainiSingh2023b}, it is proved that the polynomial maps induced by $A_1X^{k_1} + A_2Y^{k_2}$ on $M_2 (K)$, where $A_1, A_2 \in M_2(K)$ both non-zero, is surjective if and only if $A_1$ and $A_2$ can be simultaneously conjugated to a pair of matrices such that both the matrices do not have the same zero rows. 
This article proves an analogous result for the octonion algebra $\Oc(\F)$ when $A_1, A_2\in\Oc(\F)$ are two non-zero matrices. 
It is known that the automorphism group of the octonion algebra $\Oc(\F)$ is $\F$-points of groups of type $G_2$, which for simplicity we denote by $G_2(\F)$. We refer a reader to~\cite{SpringerVeldkamp2000} for the details on octonion algebra and the groups of type $G_2$. 
Note that to decide whether a polynomial map with constants is surjective, it is enough to check when the coefficients of the polynomial are chosen up to the $\Aut(\Oc(\F))$-action (which is the exceptional group of type $G_2$) on $\Oc(\F)\times\Oc(\F)$ via coordinate-wise action. We apply this strategy to obtain our results.  So, to conclude about the surjectivity of the map $A_1(X^{k_1}) + A_2(Y^{k_2})$, we only consider the tuples $(A_1, A_2)$ which are representatives of the orbits of the $G_2$-action on $\Oc(\F)^2$. The orbits under this action are described in \cite{LopatinZubkov2025}, which we make use of for our purpose. By $\mathbf{0}$, we denote the zero vector in $\F^3$. We have the following: 

\begin{mainthm}\label{mainthm}
Let $\F$ be an algebraically closed field, and $\Oc(\F)$ be the split octonion algebra over $\F$. Then, the map induced by $A_1(X^{k_1}) + A_2(Y^{k_2})$ on $\Oc(\F)$, where $A_1, A_2\in\Oc(\F)\backslash \{0\}$, is surjective if and only if the pair $(A_1, A_2)$ under $G_2(\F)$-action does not represent one of the following pairs:
\begin{multicols}{2}
\begin{enumerate}
\item $\left(\begin{pmatrix} \alpha_1 & \bf 0\\ \mathbf{0} & 0
\end{pmatrix},\begin{pmatrix} \beta_1 & \bf 0\\ \bf 0 & 0
\end{pmatrix} \right)$
\item$\left(\begin{pmatrix} 0 & \bf 0\\ \mathbf{0} & \alpha_8
\end{pmatrix},\begin{pmatrix} 0 & \bf 0\\ \bf 0 & \beta_8
\end{pmatrix} \right)$
\item $\left(\begin{pmatrix} \alpha_1 & \bf 0\\ \mathbf{0} & 0
\end{pmatrix},\begin{pmatrix} \beta_1 & (1,0,0)\\ (0,\beta_6,0) & 0
\end{pmatrix} \right)$
\item $\left(\begin{pmatrix} 0 & \bf 0\\ \mathbf{0} & \alpha_8
\end{pmatrix},\begin{pmatrix} 0 & (1,0,0)\\ (0,\beta_6,0) & \beta_8
\end{pmatrix} \right)$
\item $\left(\begin{pmatrix} 0 & (1,0,0)\\ \mathbf{0}& 0
\end{pmatrix}, \begin{pmatrix} 0 & (1,0,0)\\ \mathbf{0} & 0  \end{pmatrix} \right)$
\item  $\left(\begin{pmatrix} 0 & (1,0,0)\\ \mathbf{0}& 0 \end{pmatrix}, \begin{pmatrix} \beta_1 & (0,1,0)\\ \mathbf{0} & 0
\end{pmatrix} \right)$
\item $\left(\begin{pmatrix} 0 & (1,0,0)\\ \mathbf{0} & 0
\end{pmatrix},\begin{pmatrix} 0 & (0,1,0)\\ \mathbf{0} & \beta_8
\end{pmatrix} \right)$
\item $\left(\begin{pmatrix} 0 & (1,0,0)\\ \mathbf{0} & 0
\end{pmatrix},\begin{pmatrix} 0 & \mathbf{0}\\ (0,1,0) & 0    \end{pmatrix} \right)$
\end{enumerate}
\end{multicols}
where $\alpha_1,\alpha_8,\beta_1,\beta_6, \beta_8\in \F$. 
\end{mainthm}
The octonion algebra and the automorphism group $G_2$ are briefly recalled in \cref{sec:prelim}.
We also recall the representative of the orbit spaces of $G_2$-action on $\Oc(\F)^2$ in \cref{sec:prelim} following the work \cite{LopatinZubkov2025}. 
In \cref{sec:invert-new}, we tackle the proof of \cref{mainthm} when either $A_1$ or $A_2$ is invertible. 
The remaining of the cases for \cref{mainthm} are proved in \cref{sec:thm-2}.

\section{Representatives of the \texorpdfstring{$G_2$}{G2}-orbits on pair of split octonions}\label{sec:prelim}

In this section, we recall the octonion algebra over an algebraically closed field $\F$ of arbitrary characteristic $p\geq 0$ and their automorphism group $G_2(\F)$ to set notation for the rest of the article. We refer a reader to the excellent book \cite{SpringerVeldkamp2000} on the subject for further details. 

\subsection{The octonion algebra and its automorphism group} 

Let $\F$ be an algebraically closed field. The split octonion algebra over $\F$, denoted by $\Oc$ is
\begin{align*}
\Oc(\F)=\left\{\begin{pmatrix} \eta & \mathbf{x}\\ \mathbf{y} &\zeta \end{pmatrix}\middle| \eta,\zeta\in \F, \mathbf{x},\mathbf{y}\in \F^3\right\}.
\end{align*}
The addition is defined entry-wise, whereas the multiplication is defined using the following rule
\begin{align*}
\begin{pmatrix} \eta & \mathbf{x}\\ \mathbf{y} &\zeta
\end{pmatrix} \begin{pmatrix} \eta' & \mathbf{x}'\\     \mathbf{y}' &\zeta' \end{pmatrix} =
\begin{pmatrix} \eta\eta' + \langle \mathbf{x}, \mathbf{y}' \rangle & \eta\bx'+\zeta'\bx + \by\wedge \by'\\ 
\zeta\by'+\eta'\by + \bx\wedge \bx' &\zeta\zeta'+\langle\by,\bx'\rangle
\end{pmatrix}.
\end{align*}
Here $\langle \bx,\by\rangle = \sum\limits_{i=1}^{3}x_iy_i$, where $\bx=(x_1,x_2,x_3)$ and $\by=(y_1,y_2,y_3)$, and the exterior power is given by 
$$\langle \bx\wedge\by,\mathbf{z}\rangle=\det(\bx,\by,\mathbf{z}).$$
The set $\Oc(\F)$ is a non-associative and non-commutative algebra. The linear involution for an element $A=\begin{pmatrix} \eta & \mathbf{x}\\      \mathbf{y} &\zeta  \end{pmatrix}\in\Oc(\F)$ is given by \begin{align} \Bar{A}=\begin{pmatrix}  \zeta & -\mathbf{x}\\  -\mathbf{y} &\eta  \end{pmatrix}.
\end{align}
The norm for $A$, $N(A)=\eta\zeta-\langle\mathbf{x},\mathbf{y}\rangle$ and trace for $A$ is $T(A)=\eta+\zeta$. It turns out that over an algebraically closed field $\F$, there is a unique octonion algebra up to isomorphism, which is called split octonion (the one described above). The automorphisms of the split octonion algebra give the exceptional group of type $G_2$ (which is split), and $\F$-points of $G_2$ are denoted as $G_2(F)$. We abuse notation and write $G_2(\F)$ and $\Aut(\Oc)$  interchangeably. 

\subsection{Some Properties:}
Let $A, B\in \Oc(F)$. Then we have, 
\begin{enumerate}
 \item Power Associativity: \begin{align}
        A(A^n)=(A^n)A=A^{n+1}, (n\geq 0)
    \end{align}
\item Flexible law:
    \begin{align}
       A(BA)=(AB)A=N(A)B.
    \end{align}
\item For any two elements $A,B\in \Oc(\F)$, with $N(A)\neq 0$, one has
\begin{align}\label{eq:invert}
    A^{-1}(AB)=B,
\end{align}
\end{enumerate}

Note that the power associativity law helps us justify the computation of $A^k$ for any $A$ even though $\Oc$ is non-associative. From the alternative laws (see Lemma 1.4.2 \cite{SpringerVeldkamp2000}), we have $(AB)B = A(BB)= A(B^2)$. To understand the surjectivity of $A_1(X^{k_1}) + A_2(Y^{k_2})$ (equivalently $A_1X^{k_1} + A_2Y^{k_2}$) on $\Oc$ we may as well look at the image of $\sigma(A_1)(X^{k_1}) + \sigma(A_2)(Y^{k_2})$ where $\sigma\in \Aut(\Oc)=G_2$. Thus, it is enough to compute the images of polynomials for representatives of $(A_1, A_2)$ under the action of $\Aut(\Oc)=G_2$. The orbit representatives are determined in \cite[Theorem 4.1]{LopatinZubkov2025}), which we list here.

\begin{proposition}[Theorem 4.1, \cite{LopatinZubkov2025}] \label{pro:reprsen-simul}
Let $\F$ be an algebraically closed field. Following are the orbit representative of the $G_2(\F)$-action on $\Oc(\F)\times\Oc(\F)$, via $g\cdot(A_1,A_2)=(g(A_1),g(A_2))$;
\begin{enumerate}
\item[$\mathrm{(DD)}$] $\left(\begin{pmatrix}   \alpha_1 & \mathbf{0}\\  \mathbf{0} & \alpha_8 \end{pmatrix},
\begin{pmatrix} \beta_1 & \mathbf{0}\\ \mathbf{0} & \beta_8 \end{pmatrix}\right)$,
\item[$\mathrm{(EK_1)}$] $\left(\begin{pmatrix} \alpha_1 & \bf0\\ \bf0 & \alpha_1 \end{pmatrix},
\begin{pmatrix} \beta_1 & (1,0,0)\\ \mathbf{0} & \beta_1
\end{pmatrix}\right)$,
\item[$\mathrm{(FK)}$] $\left(\begin{pmatrix} \alpha_1 & \bf0\\ \bf0 & \alpha_8 \end{pmatrix},
\begin{pmatrix} \beta_1 & (1,0,0)\\ \mathbf{0} & \beta_8
\end{pmatrix}\right)$ with $\alpha_1\neq \alpha_8$,
\item[$\mathrm{(FN)}$] $\left(\begin{pmatrix} \alpha_1 & \bf0\\ \bf0 & \alpha_8 \end{pmatrix},
\begin{pmatrix} \beta_1 & (1,0,0)\\ (\beta_5,0,0) & \beta_8 \end{pmatrix}\right)$
with $\alpha_1\neq \alpha_8$ and $\beta_5\neq 0$,
\item[$\mathrm{(FP)}$] $\left(\begin{pmatrix}
\alpha_1 & \bf0\\ \bf0 & \alpha_8 \end{pmatrix},
\begin{pmatrix} \beta_1 & (1,0,0)\\ (0,1,0) & \beta_8   
\end{pmatrix}\right)$
with $\alpha_1\neq \alpha_8$,
\item[$\mathrm{(K_1E)}$] $\left(\begin{pmatrix} \alpha_1 & (1,0,0)\\ \bf 0 & \alpha_1 \end{pmatrix},
\begin{pmatrix} \beta_1 & \bf 0\\ \bf 0 & \beta_1
\end{pmatrix}\right)$,
\item[$\mathrm{(K_1F})$] $\left(\begin{pmatrix} \alpha_1 & (1,0,0)\\ \bf 0 & \alpha_1 \end{pmatrix},
\begin{pmatrix} \beta_1 & \bf 0\\ \bf 0 & \beta_8 \end{pmatrix}\right)$
with $\beta_1\neq \beta_8$,
\item[$\mathrm{(K_1L_1)}$] $\left(\begin{pmatrix} \alpha_1 & (1,0,0)\\ \bf 0 & \alpha_1 \end{pmatrix},
\begin{pmatrix} \beta_1 & (\beta_2,0,0)\\ \bf 0 & \beta_1
\end{pmatrix}\right)$ with $\beta_2\neq 0$,
\item[$\mathrm{(K_1L^T)}$] $\left(\begin{pmatrix} \alpha_1 & (1,0,0)\\ \bf 0 & \alpha_1 \end{pmatrix},
\begin{pmatrix} \beta_1 & \bf 0\\ (\beta_5,0,0) & \beta_8
\end{pmatrix}\right)$ with $\beta_5\neq 0$,
\item[$\mathrm{(K_1M)}$] $\left(\begin{pmatrix}
\alpha_1 & (1,0,0)\\ \bf 0 & \alpha_1 \end{pmatrix},
\begin{pmatrix} \beta_1 & (0,1,0)\\ \bf 0 & \beta_8
\end{pmatrix}\right)$,
\item[$\mathrm{(K_1M_1^T)}$] $\left(\begin{pmatrix}
\alpha_1 & (1,0,0)\\ \bf 0 & \alpha_1  \end{pmatrix},
\begin{pmatrix} \beta_1 & \bf 0\\ (0,1,0) & \beta_1
\end{pmatrix}\right)$,
\end{enumerate}
for all $\alpha_1,\alpha_2,\beta_1,\beta_2,\beta_5,\beta_8\in \F$.
\end{proposition}

\subsection{Root of some octonion elements}
We note down some formula describing the $\ell$-th power, this makes sense since the algebra $\Oc(\F)$ is power-associative, as follows:
\begin{enumerate}
\item For an element $X= \begin{pmatrix}    \alpha_X & \mathbf{b}_X\\  \mathbf{c}_X & \delta_X \end{pmatrix}\in \Oc(\F)$, when $\langle \mathbf{b}_X, \mathbf{c}_X \rangle = 0$.
We have, $X^2  = \begin{pmatrix} \alpha_X & \mathbf{b}_X\\  \mathbf{c}_X & \delta_X \end{pmatrix} \begin{pmatrix} \alpha_X & \mathbf{b}_X\\        \mathbf{c}_X & \delta_X  \end{pmatrix} = \begin{pmatrix}  \alpha_X^2 & (\alpha_X+\delta_X)\mathbf{b}_X\\ (\alpha_X+\delta_X)\mathbf{c}_X & \delta_X^2 \end{pmatrix}$.
By induction, one can prove that, for a positive integer $\ell$,
\begin{align}
\begin{pmatrix} \alpha_X & \mathbf{b}_X\\ \mathbf{c}_X & \delta_X   \end{pmatrix}^\ell = \begin{pmatrix} \alpha_X^\ell & \left(\sum\limits_{i=0}^\ell\alpha_X^i\delta_X^{\ell-i}\right) \mathbf{b}_X \\ \left(\sum\limits_{i=0}^\ell\alpha_X^i\delta_X^{\ell-i}\right)\mathbf{c}_X & \delta_X^\ell \end{pmatrix} = \begin{pmatrix} \alpha_X^\ell & f(\alpha_X,\delta_X,\ell) \mathbf{b}_X \\ f(\alpha_X,\delta_X,\ell)\mathbf{c}_X & \delta_X^\ell \end{pmatrix}
\end{align} 
where the function $f$ is given by $
f(\alpha_X,\delta_X,\ell)=\sum\limits_{i=0}^{\ell-1}\alpha_X^{i}\delta_X^{\ell-1-i}$, for all $\alpha_X,\delta_X\in\F$ and $\ell\in\mathbb{Z}_+$.
\item For an element $Y=\begin{pmatrix} \alpha_Y & \mathbf{b}_Y \\ \mathbf{c}_Y & \delta_Y \end{pmatrix}\in\Oc(\F)$, we have  $Y^{\ell}=\begin{pmatrix} \tau_1 & \tau \mathbf{b}_Y\\ \tau \mathbf{c}_Y & \tau_2 \end{pmatrix}$, where $\tau,\tau_1, \tau_2$ are functions in $\alpha_Y,\delta_Y$ and $\langle\mathbf{b}_Y, \mathbf{c}_Y\rangle$. 
\end{enumerate}

Now we have,
\begin{lemma}\label{lem:scalar-k}
Let $A = \begin{pmatrix} \alpha & \mathbf{b}\\  \mathbf{c} & \delta
\end{pmatrix}\in \Oc(\F)$ where $\alpha\neq \delta $ and $\langle \mathbf{b},\mathbf{c}\rangle = 0$. Then, for $\alpha_1\in \F^\times$, there exists $X \in \Oc(\F)$ such that $A=\alpha_1 X^{k_1}$, where $k_1\in\mathbb{Z}_{+}$.
\end{lemma}
\begin{proof}
Let $X=\begin{pmatrix} \alpha_X & \tau_1\mathbf{b} \\ \tau_2 \mathbf{c} & \delta_X \end{pmatrix}$,  where we aim to determine $\alpha_X, \delta_X, \tau_1,\tau_2\in \F$ to suit our purpose. We compute, 
\begin{align*}
A - \alpha_1 X^{k_1} & = \begin{pmatrix} \alpha & \mathbf{b}\\ \mathbf{c} & \delta \end{pmatrix} - \begin{pmatrix}\alpha_1 \alpha^{k_1}_X & \alpha_1 f(\alpha_X,\delta_X,k_1)\tau_1\mathbf{b}\\
\alpha_1 f(\alpha_X,\delta_X,k_1)\tau_2\mathbf{c} & \alpha_1 \delta^{k_2}_X  \end{pmatrix}\\ & = \begin{pmatrix} \alpha-\alpha_1\alpha^{k_1}_X & \left(1-\alpha_1  f(\alpha_X, \delta_X,k_1) \tau_1\right) \mathbf{b}\\ \left(1-\alpha_1 f(\alpha_X, \delta_X,k_1)\tau_2 \right)\mathbf{c} & \delta-\alpha_1 \delta_X^{k_1}
\end{pmatrix}.
\end{align*}
Choose $\alpha_X,\delta_X$ such that $\alpha=\alpha_1 \alpha^{k_1}_X$ and $\delta=\alpha_1 \delta_X^{k_1}$. Since $\alpha\neq \delta$, we get $\alpha_X\neq \delta_X$ and  $f(\alpha_X,\delta_X,k_1)\neq 0$. 
Further, choose $\tau_1= \tau_2 = (\alpha_1 f(\alpha_X, \delta_X, k_1))^{-1}$. It follows that $A-\alpha_1 X^{k_1}=0$.
\end{proof}

In what follows, an element $A\in \Oc(\F)$ is represented by $A=\begin{pmatrix} \alpha & \bf b\\
\bf c & \delta \end{pmatrix}$. 
Note that no restriction is imposed on the value $\langle \bf b, \bf c \rangle$ in general. For $i=1, 2, 3$, $a_i\in\F$ will denote the $i$-th entry of the vector $\bf a\in\F^3$, i.e., ${\bf a} = (a_1, a_2, a_3)$. By $e_i$, we will denote the vector with $1$ in the $i$-th place and $0$ otherwise. 
It is worth noting that the power map on $\Oc(\F)$ is not surjective, as studied in \cite{lopatin2024polynomial}.
\section{When one of the coefficients is invertible}\label{sec:invert-new}
We work over an algebraically closed field $\F$. In this section, we aim to prove the following. 
\begin{theorem}\label{coro:invertible}
    Let $A_1,A_2\in\Oc(\F)$ with $N(A_1)\neq 0$. Then for $A\in\Oc(\F)$ and positive integers $k_1,k_2\geq2$, there exist $X,Y\in\Oc(\F)$ such that \begin{align} A = A_1X^{k_1}+A_2 Y^{k_2}. \end{align}
\end{theorem}
We assume $N(A_1)\neq 0$ and $N(A_2)\in \F$. The case where $N(A_2)\neq 0$ and $N(A_1)\in \F$ can be derived by considering the pair $(A_2,A_1)$ under the action of $G_2$.\\
To prove the existence of $X$ and $Y$ for each $A\in \Oc(\F)$, we consider the following representatives of $(A_1,A_2)$ under the action of $G_2$ (note that we have combined a few classes which will be dealt together):
\begin{multicols}{2}
\begin{enumerate}
\item[(I)]  $\left(\begin{pmatrix}   \alpha_1 & \mathbf{0}\\  \mathbf{0} & \alpha_8 \end{pmatrix},
\begin{pmatrix} \beta_1 & \mathbf{0}\\ \mathbf{0} & \beta_8 \end{pmatrix}\right)$,
\item[(II)] $\left(\begin{pmatrix} \alpha_1 & \bf0\\ \bf0 & \alpha_8 \end{pmatrix},
\begin{pmatrix} \beta_1 & (1,0,0)\\ \mathbf{0} & \beta_8
\end{pmatrix}\right)$, 
\item [(III)]  $\left(\begin{pmatrix} \alpha_1 & \bf0\\ \bf0 & \alpha_8 \end{pmatrix},
\begin{pmatrix} \beta_1 & (1,0,0)\\ (\beta_5,\beta_6,0) & \beta_8 \end{pmatrix}\right)$, 
\item [(IV)] $\left(\begin{pmatrix} \alpha_1 & (1,0,0)\\ \bf 0 & \alpha_1 \end{pmatrix},
\begin{pmatrix} \beta_1 & \bf 0\\ \bf 0 & \beta_8 \end{pmatrix}\right)$,
\item [(V)] $\left(\begin{pmatrix} \alpha_1 & (1,0,0)\\ \bf 0 & \alpha_1 \end{pmatrix},
\begin{pmatrix} \beta_1 & (\beta_2,\beta_3,0)\\ \bf 0 & \beta_8
\end{pmatrix}\right)$,
\item [(VI)]$\left(\begin{pmatrix} \alpha_1 & (1,0,0)\\ \bf 0 & \alpha_1 \end{pmatrix},
\begin{pmatrix} \beta_1 & \bf 0\\ (\beta_5,\beta_6,0) & \beta_8
\end{pmatrix}\right)$.
\end{enumerate}
\end{multicols}
We note that in all the cases mentioned above, the tuples have the following conditions:
\begin{enumerate}
\item either $\beta_2\neq 0$ and $\beta_3=0$ with $\beta_1=\beta_8$, or $\beta_2= 0$ and $\beta_3=1$
\item  either $\beta_5\neq 0$ and $\beta_6=0$, or $\beta_5= 0$ and $\beta_6=1$.
\end{enumerate}
Since $N(A_1)\neq 0$, throughout this section we assume $\alpha_1,\alpha_8\in \F^\times$. We begin with some lemmas that will be used to prove the main statement of this section.  

\begin{lemma}\label{(1)-diag}
For $\beta_1,\beta_8\in\F$ not simultaneously zero and $A\in\Oc(\F)$, there exist $ Y\in\Oc(\F)$ such that $A- \begin{pmatrix} \beta_1 & \bf{0} \\ \mathbf{0} & \beta_8 \end{pmatrix}Y^{k_2}$ is either $\begin{pmatrix} \alpha_8^{-1} \alpha_1 \alpha' & \bf 0\\  \bf c & \delta \end{pmatrix}$ with $\alpha'\neq \delta$, or $\begin{pmatrix} \alpha & \bf 0\\ \bf c & \alpha_1^{-1}\alpha_8\delta' \end{pmatrix}$ with $\alpha\neq \delta'$.
\end{lemma}
\begin{proof}
We consider two cases separately, where $\beta_1\neq 0$ and $\beta_1 = 0$.
    
{\bf Case 1.} For $\beta_1\neq 0$, take $Y=\begin{pmatrix} \alpha_Y & \tau \mathbf{b}\\ 0 & 0 \end{pmatrix}$. Then, 
$$A-\begin{pmatrix} \beta_1 & \mathbf{0}\\ \mathbf{0} & \beta_8
\end{pmatrix}Y^{k_2} = \begin{pmatrix} \alpha - \beta_1 \alpha_Y^{k_2}& (1-\beta_1 f(\alpha_Y,\beta_Y,k_2) \tau)\mathbf{b}\\ \mathbf{c} & \delta \end{pmatrix}.
$$ 
Choose $\alpha_Y\neq 0$ such that $\alpha - \beta_1 \alpha_Y^{k_2}=\alpha_8^{-1}\alpha_1\alpha'$ where $\alpha'\neq \delta$. Then $f(\alpha_Y,\delta_Y,k_2)\neq 0$. Choose $\tau=(\beta_1f(\alpha_Y,\beta_Y, k_2))^{-1}$. We get, $A- \begin{pmatrix} \beta_1 & \bf{0} \\ \mathbf{0} & \beta_8 \end{pmatrix}Y^{k_2} = \begin{pmatrix}  \alpha_8^{-1} \alpha_1 \alpha' & \bf 0\\ \bf c & \delta \end{pmatrix}$ with $\alpha'\neq \delta$.

{\bf Case 2.} Now, we take $\beta_1=0$ and consider $Y=\begin{pmatrix}      0 & \mathbf{0}\\ \tau \mathbf{c} & \delta_Y \end{pmatrix}$. Then, 
$$A - \begin{pmatrix} 0 & \mathbf{0} \\ \mathbf{0} & \beta_8 \end{pmatrix}Y^{k_2} = \begin{pmatrix} \alpha & \mathbf{b} \\
(1 - \beta_8 f(\alpha_Y,\beta_Y, k_2) \tau_2 )\mathbf{c} & \delta-\beta_8 \delta_Y^{k_2}\end{pmatrix}.
$$ 
Choose $\delta_Y\neq 0$ such that $\delta - \beta_8 \delta_Y^{k_2}=\alpha_1^{-1}\alpha_8\delta'$ where $\delta'\neq \alpha$ . Then $f(\alpha_Y,\delta_Y,k_2)\neq 0$ and $\tau = (\beta_8 f(\alpha_Y,\beta_Y, k_2))^{-1}$. We get, $A- \begin{pmatrix} \beta_1 & \bf{0} \\ \mathbf{0} & \beta_8 \end{pmatrix}Y^{k_2}=\begin{pmatrix}
    \alpha & \bf 0\\
    \bf c & \alpha_1^{-1}\alpha_8\delta'
\end{pmatrix}$ with $\alpha\neq \delta'$.
\end{proof}

\begin{lemma}\label{(2)-diag}
For $\beta_1, \beta_8\in\F$ and $A\in\Oc(\F)$, there exist $ Y\in\Oc(\F)$ such that
$A-\begin{pmatrix} \beta_1 & (1,0,0) \\ \mathbf{0} & \beta_8 \end{pmatrix}Y^{k_2}$ is either $\begin{pmatrix} \alpha_8^{-1} \alpha_1 \alpha' & \bf 0\\ \mathbf{c}' & \delta' \end{pmatrix}$ with $\alpha'\neq \delta'$, or $\begin{pmatrix}  \alpha_8^{-1} \alpha_1 \alpha' & (0,b_2,b_3)\\ (c_1',0,0) & \delta'
\end{pmatrix}$ with $\alpha'\neq \delta'$.
\end{lemma}
\begin{proof}
We consider two cases, depending on $\beta_1\neq 0$ and $\beta_1=0$.

{\bf Case 1.} For $\beta_1\neq 0$, take $Y=\begin{pmatrix} \alpha_Y &  \mathbf{b}_Y\\ \bf 0 & 0
    \end{pmatrix}$. Then, 
$$A-\begin{pmatrix} \beta_1 & (1,0,0)\\ \mathbf{0} & \beta_8
\end{pmatrix}Y^{k_2} = \begin{pmatrix} \alpha - \beta_1 \alpha_Y^{k_2}& \mathbf{b} -\beta_1f(\alpha_Y,\delta_Y,k_2)\mathbf {b}_Y\\ \mathbf{c}-(1,0,0)\wedge f(\alpha_Y,\delta_Y,k_2)\mathbf{b}_Y & \delta\end{pmatrix}.
$$ 
Choose $\alpha_Y\neq 0$ such that $\alpha - \beta_1 \alpha_Y^{k_2}=\alpha_8^{-1}\alpha_1\alpha'$ where $\alpha'\neq \delta$. Then $f(\alpha_Y,\delta_Y,k_2)\neq 0$ and take $\mathbf{b}_Y$ such that $\mathbf{b} -\beta_1f(\alpha_Y,\delta_Y,k_2)\mathbf {b}_Y=0$ . We get, $A-\begin{pmatrix} \beta_1 & (1,0,0) \\ \mathbf{0} & \beta_8 \end{pmatrix}Y^{k_2}=\begin{pmatrix}
    \alpha_8^{-1}\alpha_1\alpha' & \bf 0\\
    \mathbf{c}' & \delta
\end{pmatrix}$ with $\alpha'\neq \delta$.\\

{\bf Case 2.} For the case $\beta_1=0$,  consider $Y=\begin{pmatrix}      \alpha_Y & \mathbf{b}_Y\\ \mathbf{c}_Y & \delta_Y \end{pmatrix}$ where $\mathbf{b}_Y=(0,c_3,c_2)$ and $\mathbf{c}_Y=(\tau,0,0)$. Then, 
$$A - \begin{pmatrix} 0 & (1,0,0) \\ \mathbf{0} & \beta_8 \end{pmatrix}Y^{k_2} = \begin{pmatrix} \alpha-f(\alpha_Y,\delta_Y,k_2)\langle\mathbf{c}_Y,(1,0,0)\rangle & \mathbf{b}-\delta_Y^{k_2}(1,0,0) \\
\mathbf{c} - \beta_8 f(\alpha_Y,\beta_Y, k_2)\mathbf{c}_Y-f(\alpha_Y,\beta_Y, k_2)(1,0,0)\wedge\mathbf{b}_Y & \delta-\beta_8 \delta_Y^{k_2}\end{pmatrix}.
$$ 
Choose $\delta_Y$ such that $b_1-\delta_Y^{k_2}=0$. Choose $\alpha_Y$ such that $f(\alpha_Y,\delta_Y,k_2)=1$ and hence $\tau$ such that $\alpha-\tau=\alpha_8^{-1}\alpha_1\alpha'$ where $\alpha'\neq \delta-\beta_8b_1$. We get, $A-\begin{pmatrix} \beta_1 & (1,0,0) \\ \mathbf{0} & \beta_8 \end{pmatrix}Y^{k_2}=\begin{pmatrix}
    \alpha_8^{-1}\alpha_1\alpha' & (0,b_2,b_3)\\
    (c_1-\beta_8\tau,0,0) & \delta-\beta_8\delta_Y^{k_2}
\end{pmatrix}$, as desired.
\end{proof}

\begin{lemma}\label{(3)-diag}
For $\beta_1,\beta_8\in \F$ and $A\in\Oc(\F)$, there exist $X, Y\in\Oc(\F)$ such that
\begin{align*}A- \begin{pmatrix} \beta_1 & (1,0,0) \\ (\beta_5,\beta_6,0) & \beta_8 \end{pmatrix}Y^{k_2}=\begin{pmatrix}
    \alpha_8^{-1}\alpha_1\alpha' & (0,b_2',b_3')\\
    (c_1',0,0) & \delta'
\end{pmatrix} \end{align*} where $\alpha'\neq \delta'$.
\end{lemma}
\begin{proof}
We consider two cases, depending on $\beta_5\neq 0, \beta_6=0$ and $\beta_5=0,\beta_6=1$.

{\bf Case 1.} For $\beta_5\neq 0,\beta_6=0$, take $Y=\begin{pmatrix} \alpha_Y &  \mathbf{b}_Y\\ \mathbf{c}_Y & \delta_Y
    \end{pmatrix}$ where $\mathbf{b}_Y=(0,c_3,c_2)$ and $\mathbf{c}_Y=(\tau,0,0)$. Then, 
$A-\begin{pmatrix} \beta_1 & (1,0,0) \\ (\beta_5,\beta_6,0) & \beta_8 \end{pmatrix}Y^{k_2}$ is  $$ \begin{pmatrix} \alpha - \beta_1 \alpha_Y^{k_2}-f(\alpha_Y,\delta_Y,k_2)\langle\mathbf{c}_Y,(1,0,0)\rangle& \mathbf{b} -\beta_1f(\alpha_Y,\delta_Y,k_2)\mathbf {b}_Y-\delta_Y^{k_2}(1,0,0)\\ \mathbf{c}-(1,0,0)\wedge f(\alpha_Y,\delta_Y,k_2)\mathbf{b}_Y-\alpha_Y^{k_2}(\beta_5,0,0)-\beta_8f(\alpha_Y,\delta_Y,k_2)\mathbf{c}_Y & \delta-\beta_8\delta_Y^{k_2}\end{pmatrix}.
$$
Choose $\delta_Y$ such that $b_1-\delta_Y^{k_2}=0$. Choose $\alpha_Y$ such that $f(\alpha_Y,\delta_Y,k_2)=1$ and hence $\tau$ such that $\alpha - \beta_1 \alpha_Y^{k_2}-\tau=\alpha_8^{-1}\alpha_1\alpha'$ where $\alpha'\neq \delta-\beta_8b_1$. We get, $$A- \begin{pmatrix} \beta_1 & (1,0,0) \\ (\beta_5,\beta_6,0) & \beta_8 \end{pmatrix}Y^{k_2}=\begin{pmatrix}
    \alpha_8^{-1}\alpha_1\alpha' & (0,b_2',b_3')\\
    (c_1',0,0) & \delta-\beta_8\delta_Y^{k_2}
\end{pmatrix}$$ where $b_2'=b_2-\beta_1c_3,b_3'=b_3-\beta_1c_2$ and $c_1'=c_1-\beta_5\alpha_Y^{k_2}-\beta_8\tau$.\\

{\bf Case 2.} For $\beta_5= 0,\beta_6=1$, take $Y=\begin{pmatrix} \alpha_Y &  \mathbf{b}_Y\\ \mathbf{c}_Y & \delta_Y
    \end{pmatrix}$ where $\mathbf{b}_Y=(0,c_2',c_3)$ and $\mathbf{c}_Y=(\tau,0,0)$. Then, 
$A-\begin{pmatrix} \beta_1 & (1,0,0) \\ (0,1,0) & \beta_8 \end{pmatrix}Y^{k_2}= \begin{pmatrix} \substack{\alpha - \beta_1 \alpha_Y^{k_2}\\-f(\alpha_Y,\delta_Y,k_2)\langle\mathbf{c}_Y,(1,0,0)\rangle}& \substack{\mathbf{b} -\beta_1f(\alpha_Y,\delta_Y,k_2)\mathbf {b}_Y\\-\delta_Y^{k_2}(1,0,0)-(0,1,0)\wedge \mathbf{c}_Y}\\&\\ \substack{\mathbf{c}-f(\alpha_Y,\delta_Y,k_2)((1,0,0)\wedge \mathbf{b}_Y)\\-\alpha_Y^{k_2}(0,1,0)-\beta_8\mathbf{c}_Y} & \substack{\delta-\beta_8\delta_Y^{k_2}-\\f(\alpha_Y,\delta_Y,k_2)\langle(0,1,0),\mathbf{b}_Y\rangle}\end{pmatrix}.
$
Choose $\delta_Y$ such that $b_1-\delta_Y^{k_2}=0$. Choose $\alpha_Y$ such that $f(\alpha_Y,\delta_Y,k_2)=1$. Choose $c_2'=c_2-\alpha_Y^{k_2}$. Then $\tau$ such that $\alpha - \beta_1 \alpha_Y^{k_2}-\tau=\alpha_8^{-1}\alpha_1\alpha'$ where $\alpha'\neq \delta-c_2'-\beta_8b_1$. We get, $$A- \begin{pmatrix} \beta_1 & (1,0,0) \\ (\beta_5,\beta_6,0) & \beta_8 \end{pmatrix}Y^{k_2}=\begin{pmatrix}
    \alpha_8^{-1}\alpha_1\alpha' & (0,b_2',b_3')\\
    (c_1',0,0) & \delta-c_2'-\beta_8\delta_Y^{k_2}
\end{pmatrix}$$ where $b_2'=b_2-\beta_1c_3,b_3'=b_3-\beta_1c_2+\tau$ and $c_1'=c_1-\beta_8\tau$. Note that in the second case, the above proof works for any non-zero $\beta_6$.
\end{proof}

\begin{lemma}\label{conjugate-diag}
For $\alpha_1,\alpha_8\in \F^\times$ and $A=\begin{pmatrix} \alpha_8^{-1} \alpha_1 \alpha' & \mathbf b\\ \mathbf c & \delta'
\end{pmatrix}$ such that $\alpha'\neq \delta'$ and $\langle \mathbf{b},\mathbf c \rangle = 0$. Then there always exists an $X\in \Oc(\F)$ such that $A=\begin{pmatrix} \alpha_1 & \bf 0\\
\bf 0 & \alpha_8 \end{pmatrix}X^{k_1}$.
\end{lemma}
\begin{proof}
 We make use of \cref{eq:invert}. Multiplying by conjugate of $\begin{pmatrix}
    \alpha_1 & \bf 0\\
    \bf 0 & \alpha_8
\end{pmatrix}$ from left in the given equation, we get $\begin{pmatrix}
    \alpha_1\alpha' & \alpha_8\mathbf{b}\\
    \alpha_1\mathbf c & \alpha_1\delta'
\end{pmatrix}=\alpha_1\alpha_8X^{k_1}.$ By \cref{lem:scalar-k}, there exist $X\in \Oc(\F)$ such that $A-\begin{pmatrix}
    \alpha_1 & \bf 0\\
    \bf 0 & \alpha_8
\end{pmatrix}X^{k_1}=0$.
\end{proof}
\begin{theorem}
Let $A_1,A_2\in\Oc(\F)\smallsetminus\{0\}$ with $N(A_1)\neq 0$ and $A_1$ has representative $\begin{pmatrix}\alpha_1 & \bf 0\\
\bf 0 & \alpha_8 \end{pmatrix}$ under $G_2$ action. Then for $A\in\Oc(\F)$ and positive integers $k_1,k_2\geq 2$, there exist $X,Y\in\Oc(\F)$ such that $$A=A_1X^{k_1}+A_2Y^{k_2}.$$
\end{theorem}
\begin{proof}
The representatives of $(A_1,A_2)$ under the simultaneous action of $G_2$ are given by
\begin{multicols}{2}
        \begin{enumerate}
        \item $\left(\begin{pmatrix}   \alpha_1 & \mathbf{0}\\  \mathbf{0} & \alpha_8 \end{pmatrix},
\begin{pmatrix} \beta_1 & \mathbf{0}\\ \mathbf{0} & \beta_8 \end{pmatrix}\right)$,
\item$\left(\begin{pmatrix} \alpha_1 & \bf0\\ \bf0 & \alpha_8 \end{pmatrix},
\begin{pmatrix} \beta_1 & (1,0,0)\\ \mathbf{0} & \beta_8
\end{pmatrix}\right)$ 
\item  $\left(\begin{pmatrix} \alpha_1 & \bf0\\ \bf0 & \alpha_8 \end{pmatrix},
\begin{pmatrix} \beta_1 & (1,0,0)\\ (\beta_5,\beta_6,0) & \beta_8 \end{pmatrix}\right)$.
    \end{enumerate}
\end{multicols}
By \cref{(1)-diag}, \cref{(2)-diag} and \cref{(3)-diag}, $A-A_2Y^{k_2}=\begin{pmatrix} \alpha_8^{-1}\alpha_1\alpha' & \mathbf{b}'\\ \mathbf{c}' & \delta' \end{pmatrix}$ such that $\alpha'\neq \delta'$ and $\langle\mathbf{b}',\mathbf{c}'\rangle=0$ in each case. By \cref{conjugate-diag}, there exist $X\in\Oc(\F)$ such that $A=A_1X^{k_1}+A_2Y^{k_2}$. 
\end{proof}

\begin{proposition}
    Let $A\in\Oc(\F)$ and $\beta_1,\beta_8\in \F$ not simultaneously zero. Then there exists $X,Y\in\Oc(\F)$ such that 
    $$A=\begin{pmatrix}
        \alpha_1 & (1,0,0)\\
        \bf 0 & \alpha_1
    \end{pmatrix}X^{k_1}+\begin{pmatrix}
        \beta_1 & \bf 0\\
        \bf 0 & \beta_8
    \end{pmatrix}Y^{k_2}.$$
\end{proposition}
\begin{proof}
It is enough to consider the case $\beta_1\beta_8=0$. When $\beta_8=0$ and $\beta_1\neq 0$, choose $Y=\begin{pmatrix}
        \alpha_Y & \mathbf{b}_Y\\
        \mathbf{0} & 0
    \end{pmatrix}$, which gives
\begin{align*}
    A-\begin{pmatrix}
        \beta_1 & \mathbf{0}\\
        \mathbf{0} & 0 
    \end{pmatrix}Y^{k_2}=\begin{pmatrix}
    \alpha-\beta_1\alpha_Y^{k_2} & \mathbf{b}-\beta_1\alpha_Y^{k_2-1}\mathbf{b}_Y\\
    \mathbf{c} & \delta
    \end{pmatrix}.
\end{align*}
Choose $\alpha_Y\neq 0$ such that $\alpha_1(\alpha-\beta_1\alpha_Y^{k_2})-c_1\neq \alpha_1\delta$ and $\mathbf{b}_Y$ such that $\mathbf{b}'=\mathbf{b}-\beta_1\alpha_Y^{k_2-1}\mathbf{b}_Y$ satisfies $\alpha_1\mathbf{b}'-\delta(1,0,0)=\mathbf{0}$. 
Then using \cref{eq:invert} we need to find solution for $X$ such that $N(A_1)X^{k_1}=\begin{pmatrix}
    \alpha_1(\alpha-\beta_1\alpha_Y^{k_2})-c_1 & \mathbf{0}\\\alpha_1\mathbf{c} & \alpha_1\delta
\end{pmatrix}$, 
which exists by \cref{lem:scalar-k}. 

In case $\beta_1=0$ and $\beta_8\neq0$, choose $Y=\begin{pmatrix}
    0 & \mathbf{0}\\
    \mathbf{c}_Y & \delta_Y
\end{pmatrix}$ where $\delta_Y\neq0$ and satisfies $\alpha\neq \delta-\beta_8\delta_Y^{k_2}$, and $\mathbf{c}_Y$ satisfies $\mathbf{c}-\beta_8\delta_Y^{k_2-1}\mathbf{c}_Y=(1,0,0)\wedge\mathbf{b}$. Then 
\begin{align*}
    A-\begin{pmatrix}
        0 & \mathbf{0}\\
        \mathbf{0} & \beta_8
    \end{pmatrix}Y^{k_2}=\begin{pmatrix}
        \alpha & \mathbf{b}\\
        (1,0,0)\wedge\mathbf{b} & \delta-\beta_8\delta_Y^{k_2}
    \end{pmatrix}.
\end{align*}
Then again using \cref{eq:invert} we need existence of $X$ such that $N(A_1)X^{k_1}=\begin{pmatrix}
    \alpha_1\alpha &\mathbf{b}-(\delta-\beta_8\delta_Y^{k_2},0,0)\\\mathbf{0} & \alpha_1(\delta-\beta_8\delta_Y^{k_2})
\end{pmatrix}$.
Such an $X$ exists because of  \cref{lem:scalar-k}.
\end{proof}
\begin{proposition} Let $A\in\Oc(\F)$, $\beta_i\in \F$, for $i=1,2,3,8$. Then there exist $X,Y\in\Oc(\F)$ such that
    $$A=\begin{pmatrix} \alpha_1 & (1,0,0)\\ \bf 0 & \alpha_1 \end{pmatrix}X^{k_1}+
\begin{pmatrix} \beta_1 & (\beta_2,\beta_3,0)\\ \bf 0 & \beta_8
\end{pmatrix}Y^{k_2}$$ 
when one of the following holds:
\begin{enumerate}
    \item $\beta_2\neq 0$, $\beta_3=0$ and $\beta_1=\beta_8$ , or 
    \item $\beta_3=1$ and $\beta_2=0$.
\end{enumerate}
\end{proposition}
\begin{proof}
    Consider the first case, where $\beta_2\neq 0$, $\beta_3=0$ and $\beta_1=\beta_8$. \\
    {\bf{Case 1.}} If $\beta_1\neq 0$. Let $Y=\begin{pmatrix}
        \alpha_Y & \mathbf{b}_Y\\
        \bf 0 & 0
    \end{pmatrix}$ where $\alpha_Y\neq 0$ and is such that $\alpha_1(\alpha-\beta_1\alpha_Y^{k_2})-c_1\neq \alpha_1\delta$ and $\mathbf{b}_Y$ satisfies $\alpha_1(b-\beta_1\alpha_Y^{k_2-1}\mathbf{b}_Y)-\delta(1,0,0)=0$. Then $$A-A_2Y^{k_2}=\begin{pmatrix}
        \alpha-\beta_1\alpha_Y^{k_2} & \alpha_1^{-1}\delta(1,0,0)\\
        (\beta_2,0,0)\wedge\mathbf{b}_Y & \delta
    \end{pmatrix}=A_1X^{k_1}.$$
    Using \cref{eq:invert}, we get 
    $$\begin{pmatrix}
        \alpha_1(\alpha-\beta_1\alpha_Y^{k_2})-c_1 & \bf 0\\
        \mathbf{c}' & \alpha_1\delta
    \end{pmatrix}=N(A_1)X^{k_1}.$$
    {\bf{Case 2.}} If $\beta_1=0$. Let $\delta_Y$ be such that $b_1-\delta=\delta_Y^{k_2}$ and $\alpha_Y$ be such that $f(\alpha_Y,\delta_Y,k_2)=1$.  Let $Y=\begin{pmatrix}
        \alpha_Y & \mathbf{b}_Y\\
        \mathbf{c}_Y  & \delta_Y
    \end{pmatrix}$. Here, $\mathbf{c}_Y=((c_Y)_1,0,0)$ and $(c_Y)_1$ is such that $\alpha_1(\alpha-(c_Y)_1\beta_2)-c_1\neq \alpha_1\delta$ and $\mathbf{b}_Y=(0,(b_Y)_2,(b_Y)_3)$ satisfies $\alpha_1(\mathbf{c}-\beta_2\mathbf{b}_Y)-(1,0,0)\wedge\mathbf{b}=(c_1,0,0)$. Then $$A-A_2Y^{k_2}=\begin{pmatrix}
        \alpha-(c_Y)_1\beta_2 & (\delta,b_2,b_3)\\
        \mathbf{c}-(\beta_2,0,0)\wedge\mathbf{b}_Y & \delta
    \end{pmatrix}=A_1X^{k_1}.$$
    Using \cref{eq:invert}, we get 
    $$\begin{pmatrix}
        \alpha_1(\alpha-(c_Y)_1\beta_2)-c_1 & (0,b_2',b_3')\\
        (c_1,0,0) & \alpha_1\delta
    \end{pmatrix}=N(A_1)X^{k_1}.$$
 Consider the second case, $\beta_3=1$ and $\beta_2=0$.\\
 Let $\delta_Y$ such that $b_2=\delta_Y^{k_2}$ and $\alpha_Y$ such that $f(\alpha_Y,\delta_Y,k_2)=1$. Let $Y=\begin{pmatrix}
     \alpha_Y & \mathbf{b}_Y\\
     \mathbf{c}_Y & \delta_Y
 \end{pmatrix}$ where $\mathbf{b}_Y=(c_3,0,c_1)$ and $\mathbf{c}_Y=(0,(c_Y)_2,0)$ such that $\alpha-\beta_1\alpha_Y^{k_2}-(c_Y)_2\neq \delta-\beta_8b_2$. Then $$A-A_2Y^{k_2}=\begin{pmatrix}
        \alpha-\beta_1\alpha_Y^{k_2}-(c_Y)_2 & (b_1-\beta_1c_3,0,b_3-\beta_1c_1)\\
        (0,c_2-\beta_8(c_Y)_2,0) & \delta-\beta_8b_2
    \end{pmatrix}=A_1X^{k_1}.$$
    Using \cref{eq:invert}, we get 
    $$\begin{pmatrix}
        \alpha_1(\alpha-\beta_1\alpha_Y^{k_2}-(c_Y)_2) & \alpha_1(b_1-\beta_1c_3,0,b_3-\beta_1c_1)-(\delta-\beta_8b_2,0,0)\\
      (0,\alpha_1(c_2-\beta_8(c_Y)_2)+b_3-\beta_1c_1,0)  & \alpha_1(\delta-\beta_8b_2)
    \end{pmatrix}=N(A_1)X^{k_1}.$$
    In all the cases discussed above, we have 
    $$\begin{pmatrix}
        \alpha' & \mathbf{b}'\\
        \mathbf{c}' & \delta'
    \end{pmatrix}=N(A_1)X^{k_1}$$ where $\alpha'\neq \delta'$ and $\langle\mathbf{b}',\mathbf{c}'\rangle=0$. By \cref{lem:scalar-k}, there exist $X$ such that $A=A_1X^{k_1}+A_2Y^{k_2}$.
\end{proof}

\color{black}
\begin{proposition}
    Let $A\in\Oc(\F)$ and $\beta_1,\beta_8\in \F$. Then there exists $X,Y\in\Oc(\F)$ such that 
    $$A=\begin{pmatrix}
        \alpha_1 & (1,0,0)\\
        \bf 0 & \alpha_1
    \end{pmatrix}X^{k_1}+\begin{pmatrix}
        \beta_1 & \bf 0\\
        (\beta_5,\beta_6,0)  & \beta_8
    \end{pmatrix}Y^{k_2}$$
    where either $\beta_5\neq 0$ and $\beta_6=0$ or $\beta_5=0$ and $\beta_6=1$.
\end{proposition}
\begin{proof}
    Consider the following two cases:\\
    {\bf{Case 1.}} When $\beta_5\neq 0$ and $\beta_6=0$. Let $\alpha_Y^{k_1}=\beta_5^{-1}c_1$ and $\delta_Y$ be such that $f(\alpha_Y,\delta_Y,k_2)=1$. For $Y=\begin{pmatrix}
        \alpha_Y & \mathbf{b}_Y\\
        \mathbf{c}_Y & \delta_Y
    \end{pmatrix}$ where $\mathbf{c}_Y=\beta_5^{-1}(0,b_3,b_2)$ and $\mathbf{b}_Y=(\tau,0,0)$ such that $\delta-\beta_5\tau-\beta_8\delta_Y^{k_2}\neq \alpha-\beta_1\beta_5^{-1}c_1$, we get 
    $$\begin{pmatrix}
        \alpha' & (b_1-\beta_1\tau,0,0)\\
        (0,c_2-\beta_8\beta_5^{-1}b_3,c_3-\beta_8\beta_5^{-1}b_2) & \delta'
    \end{pmatrix}=A-A_2Y^{k_2}=A_1X^{k_1}$$
   where $\alpha'=\alpha-\beta_1\alpha_Y^{k_2}$, $\delta'=\delta-\beta_5\tau-\beta_8\delta_Y^{k_2}$ and $\alpha'\neq \delta'$. Multiplying the above equation by conjugate of $A_1$, we get 
   $$\begin{pmatrix}
       \alpha_1\alpha' & (\alpha_1(b_1-\beta_1\tau)-\delta',0,0)\\
        \alpha_1(0,c_2-\beta_8\beta_5^{-1}b_3,c_3-\beta_8\beta_5^{-1}b_2) & \alpha_1\delta'
   \end{pmatrix}=N(A_1)X^{k_1}.$$ By using \cref{lem:scalar-k}, there exist $X\in\Oc(\F)$ such that $A=A_1X^{k_1}+A_2Y^{k_2}$.\\

   {\bf{Case 2.}} When $\beta_5= 0$ and $\beta_6=1$. Let $\alpha_Y^{k_2}=c_2$ and $\delta_Y$ be such that $f(\alpha_Y,\delta_Y,k_2)=1$. For $Y=\begin{pmatrix}
        \alpha_Y & \mathbf{b}_Y\\
        \mathbf{c}_Y & \delta_Y
    \end{pmatrix}$ where $\mathbf{c}_Y=(b_3,0,b_2-\tau_2)$ and $\mathbf{b}_Y=(0,\tau_1,0)$ such that $\tau_1$ satisfies $\alpha_1(\delta-\tau_1-\beta_8\delta_Y^{k_2})\neq \alpha_1(\alpha-\beta_1c_2)-c_1\beta_8b_3$ and $\tau_2=\alpha_1^{-1}(\delta-\tau_1-\beta_8\delta_Y^{k_2})$. We get 
    $$\begin{pmatrix}
        \alpha' & (\tau',b_2-\beta_1\tau_1,0)\\
        (c_1-\beta_8b_3,0,c_3-\beta_8(b_1-\tau_2)) & \delta'
    \end{pmatrix}=A-A_2Y^{k_2}=A_1X^{k_1}$$
   where $\alpha'=\alpha-\beta_1c_2$, $\delta'=\delta-\tau_1-\beta_8\delta_Y^{k_2}$. Multiplying the above equation by conjugate of $A_1$, we get 
   $$\begin{pmatrix}
       \alpha_1\alpha'-c_1+\beta_8b_3 & (0,b_2-\beta_1\tau,0)\\
        \alpha_1(c_1-\beta_8b_3,0,c_3-\beta_8(b_1+\tau_2)+b_1\tau) & \alpha_1\delta'
   \end{pmatrix}=N(A_1)X^{k_1}.$$  Since $\alpha_1\alpha'-c_1+\beta_8b_3\neq \alpha_1\delta'$, using \cref{lem:scalar-k}, there exist $X\in\Oc(\F)$ such that $A=A_1X^{k_1}+A_2Y^{k_2}$.
\end{proof}

\section{When both of the coefficients are non-invertible}\label{sec:thm-2}
In the section, we look at the case when both $A_1$ and $A_2$ are non-unit. We will consider the cases depending on orbit representatives.

\subsection{\texorpdfstring{$A_1$ is non-unit and is diagonal}{A1 is non-unit and is diagonal}}
We first consider the case when $A$ is a non-unit and diagonal.

\begin{lemma}\label{lem: diag-2}
Let $\begin{pmatrix} \alpha & \mathbf{b}\\ \mathbf{0} & 0   \end{pmatrix}{\color{black}\in\Oc(\F)}$ and $\alpha_1\in\F^\times$. Then for a positive integer $k$, there exists $X\in \Oc(\F)$ such that $$\begin{pmatrix}       \alpha & \mathbf{b}\\ \mathbf{0} & 0 \end{pmatrix} = \begin{pmatrix} \alpha_1 & \mathbf{0}\\ \mathbf{0} & 0 \end{pmatrix}X^{k}.$$
\end{lemma}
\begin{proof} Let $X=\begin{pmatrix} \alpha_X & \mathbf{b} \\      \mathbf{0} & \delta_X \end{pmatrix}$. Then, 
$\begin{pmatrix} \alpha_1 & \mathbf{0}\\ \mathbf{0} & 0 \end{pmatrix}X^{k}  = \begin{pmatrix} \alpha_1 & \mathbf{0}\\ \mathbf{0} & 0 \end{pmatrix} \begin{pmatrix} \alpha^{k}_X & f(\alpha_X,\delta_X,k)\mathbf{b}\\ \mathbf{0} & \delta^{k}_X
\end{pmatrix} = \begin{pmatrix} \alpha_1  \alpha^{k}_X & \alpha_1 f(\alpha_X,\delta_X,k)\mathbf{b}\\ \mathbf{0} & 0 \end{pmatrix}.
$
Let $\alpha_X$ be such that $\alpha = \alpha_1\alpha^{k}_X$. Choose $\delta_X$ such that $f(\alpha_X,\delta_X, k)= \alpha_1^{-1}$. This gives the required element $X$.
\end{proof}
\begin{lemma}\label{lem: diag-3}
Let $\begin{pmatrix} 0 & \mathbf{0}\\ \mathbf{c} & \delta   \end{pmatrix}{\color{black}\in\Oc(\F)}$ and $\alpha_8\in\F^\times$. Then for a positive integer $k$, there exists $Y\in \Oc(\F)$ such that $$\begin{pmatrix} 0 & \mathbf{0}\\ \mathbf{c} & \delta   \end{pmatrix} = \begin{pmatrix} 0 & \mathbf{0}\\ \mathbf{0} & \alpha_8 \end{pmatrix}Y^{k}.$$
\end{lemma}
\begin{proof} Let $Y=\begin{pmatrix} \alpha_Y & \mathbf{0} \\      \mathbf{c} & \delta_Y \end{pmatrix}$. Then, 
$$\begin{pmatrix} 0 & \mathbf{0}\\ \mathbf{0} & \alpha_8 \end{pmatrix}Y^{k}  = \begin{pmatrix} 0 & \mathbf{0}\\ \mathbf{0} & \alpha_8 \end{pmatrix} \begin{pmatrix} \alpha^{k}_Y & \mathbf{0}\\ f(\alpha_Y,\delta_Y,k)\mathbf{c} & \delta^{k}_Y
\end{pmatrix} = \begin{pmatrix} 0 & \mathbf{0}\\ \alpha_8f(\alpha_Y,\delta_Y,k)\mathbf{c} & \alpha_8  \alpha^{k}_Y \end{pmatrix}.
$$
Let $\delta_Y$ be such that $\delta= \alpha_8\delta^{k}_Y$. Choose $\alpha_Y$ such that $f(\alpha_Y,\delta_Y, k)= \alpha_8^{-1}$. This gives the required element $Y$.
\end{proof}

\begin{proposition}\label{prop:diag}
    Let $\alpha_i,\beta_i\in\F$, where $i=1,8$ and $A\in\Oc(\F)$. Then there exist $X,Y\in\Oc(\F)$ such that 
    $$A=\begin{pmatrix}
        \alpha_1 & \mathbf{0}\\
        \mathbf{0} & \alpha_8
    \end{pmatrix}X^{k_1}+\begin{pmatrix}
        \beta_1 & \mathbf{0}\\
        \mathbf{0} & \beta_8
    \end{pmatrix}$$ if and only if $\alpha_1\beta_8\in\F^\times$ or $\alpha_8\beta_1\in\F^\times$.
\end{proposition}
\begin{proof}
    If $\alpha_1\beta_8\in\F^\times$ or $\alpha_8\beta_1\in\F^\times$, then the proof follows from \cref{lem: diag-2} and \cref{lem: diag-3}. Suppose $\alpha_1=\beta_1=0$ or $\alpha_8=\beta_8=0$. Then $X,Y$ exist only if $A$ is of the form $\begin{pmatrix}
        \alpha & \mathbf{b}\\
        \mathbf{0} & 0
    \end{pmatrix}$ or $\begin{pmatrix}
        0 & \mathbf{0}\\
        \mathbf{c} & \delta
    \end{pmatrix}$
    depending on the case.
\end{proof}
\begin{proposition}\label{prop:diag-upper}
For $\alpha_1,\beta_8\in \F^\times$ and $A\in\Oc(\F)$, there exist $X, Y\in \Oc(\F)$ such that 
\begin{align}\label{eqn: diag-1}A=
\begin{pmatrix} \alpha_1 & \mathbf{0}\\ \mathbf{0} & 0 \end{pmatrix}X^{k_1} + \begin{pmatrix} 0 & (1,0,0)\\ \mathbf{0} & \beta_8 \end{pmatrix}Y^{k_2} .
\end{align}
\end{proposition}
\begin{proof}  Let $Y = \begin{pmatrix} \alpha_Y & \mathbf{0}\\ \mathbf{c} & \delta_Y   \end{pmatrix}$. Then,  
$$A- \begin{pmatrix} 0 & (1,0,0)\\ \mathbf{0} & \beta_8 \end{pmatrix}Y^{k_2}  = \begin{pmatrix} \alpha & \mathbf{b}\\ \mathbf{c} & \delta \end{pmatrix} - \begin{pmatrix} \langle(1,0,0), f(\alpha_Y,\delta_Y,k_2) \mathbf{c} \rangle & \delta_Y^{k_2} (1,0,0)\\
\beta_8 f(\alpha_Y,\delta_Y,k_2)\mathbf{c} & \beta_8\delta_Y^{k_2}   \end{pmatrix}.$$
Let $\delta_Y$ be such that $\delta = \beta_8\delta_Y^{k_2}$. Choose $\alpha_Y$ such that $f(\alpha_Y,\delta_Y, k_2) = \beta_8^{-1}$. So, 
$$A- \begin{pmatrix} 0 & (1,0,0)\\ \mathbf{0} & \beta_8    \end{pmatrix}Y^{k_2} = \begin{pmatrix} \alpha' & \mathbf{b}'\\ \mathbf{0} & 0 \end{pmatrix},$$ 
for some $\alpha'\in\F$ and ${\bf b}'\in\F^3$. Now \cref{lem: diag-2} is applicable, and we get the result.
 \end{proof}

 \begin{proposition}\label{prop:diag2-upper}
For $\beta_1,\alpha_8\in \F^\times$ and $A\in\Oc(\F)$. There exists $X, Y\in \Oc(\F)$ such that 
\begin{align}\label{eqn: diag-2}A=
\begin{pmatrix} 0 & \mathbf{0}\\ \mathbf{0} & \alpha_8 \end{pmatrix}X^{k_1} + \begin{pmatrix} \beta_1 & (1,0,0)\\ \mathbf{0} & \beta_8 \end{pmatrix}Y^{k_2} .
\end{align}
\end{proposition}
\begin{proof} Let $\alpha_Y$ be such that $\alpha-\beta_1\alpha_Y^{k_2} =0$ and $\delta_Y$ such that $f(\alpha_Y,\delta_Y,k_2)=1$. Now,   let $Y = \begin{pmatrix} \alpha_Y & \mathbf{b}_Y\\ \bf 0 & \delta_Y   \end{pmatrix}$. Then,  
$$A- \begin{pmatrix} \beta_1 & (1,0,0)\\ \mathbf{0} & \beta_8 \end{pmatrix}Y^{k_2}  = \begin{pmatrix} 0 & \mathbf{b}-\beta_1\mathbf{b}_Y-\delta_Y^{k_2}(1,0,0)\\ \mathbf{c}-(1,0,0)\wedge\mathbf{b}_Y & \delta-\beta_8\delta_Y^{k_2} \end{pmatrix} .$$
Choose $\mathbf{b}_Y$ such that $\mathbf{b}-\beta_1\mathbf{b}_Y-\delta_Y^{k_2}(1,0,0)=0$. So, 
$$A- \begin{pmatrix} \beta_1 & (1,0,0)\\ \mathbf{0} & \beta_8    \end{pmatrix}Y^{k_2} = \begin{pmatrix} 0 & \mathbf{0}\\ \mathbf{c}' & \delta' \end{pmatrix},$$ 
for some $\delta'\in\F$ and ${\bf c}'\in\F^3$. Now \cref{lem: diag-3} is applicable, and we get the result.
 \end{proof}

\begin{proposition}\label{prop:diag-beta5}
For $\alpha_1,\beta_5, \beta_8 \in\F^\times$ and $A\in\Oc(\F)$, there  exist $X$ and $Y$ in $\Oc(\F)$ such that 
\begin{enumerate}
\item $$A= \begin{pmatrix}\alpha_1 & \mathbf{0}\\ \mathbf{0} & 0 \end{pmatrix}X^{k_1} + \begin{pmatrix} \beta_1 & (1,0,0) \\ (\beta_5,0,0) & \beta_8 \end{pmatrix}Y^{k_2}.$$
\item \begin{equation}\label{eqn:diag-3}
    A=\begin{pmatrix} \alpha_1 & \mathbf{0}\\ \mathbf{0} & 0    \end{pmatrix}X^{k_1} + \begin{pmatrix} \beta_1 & (1,0,0)\\    (0,1,0) & \beta_8 \end{pmatrix}Y^{k_2}.
\end{equation}
\end{enumerate}
\end{proposition}
\begin{proof}
(1)  Since $\begin{pmatrix} \beta_1 & (1,0,0)\\ (\beta_5,0,0) & \beta_8 \end{pmatrix}$ is a singular element of $\Oc(\F)$, we have $\beta_1 \beta_8 = \beta_5\neq 0$. Take $Y = \begin{pmatrix} \alpha_Y & \mathbf{0}\\ \mathbf{c}_Y & \delta_Y \end{pmatrix}$ then, 
\begin{align*}
A-\begin{pmatrix} \beta_1 & (1,0,0)\\ (\beta_5,0,0) & \beta_8    \end{pmatrix}Y^{k_2}  = 
\begin{pmatrix}
\alpha - \beta_1\alpha_Y^{k_2} - \langle(1,0,0), f(\alpha_Y, \delta_Y, k_2)\mathbf{c}_Y\rangle & \mathbf{b} - \delta_Y^{k_2}(1,0,0)\\ \mathbf{c}- \alpha_Y^{k_2}(\beta_5,0,0) -\beta_8 f(\alpha_Y, \delta_Y, k_2)\mathbf{c}_Y & \delta- \beta_8\delta_Y^{k_2}
\end{pmatrix}.
\end{align*}
Choose $\delta_Y$ such that $\delta - \delta_Y^{k_2}=0$. Let $\alpha_Y$ such that $f(\alpha_Y,\delta_Y, k_2) = \beta_8^{-1}$. Then $\mathbf{c}_Y = \mathbf{c} - \alpha_Y^{k_2}(\beta_5,0,0)$. We have 
$A-\begin{pmatrix} \beta_1 & (1,0,0)\\ (\beta_5,0,0) & \beta_8      \end{pmatrix}Y^{k_2} = \begin{pmatrix} \alpha' & \mathbf{b}'\\       \mathbf{0} & 0 \end{pmatrix}$. By \cref{lem: diag-2}, there exist $X\in \Oc(\F)$ such that $\begin{pmatrix} \alpha' & \mathbf{b}'\\ \mathbf{0} & 0      \end{pmatrix} = \begin{pmatrix} \alpha_1 & \mathbf{0} \\ \mathbf{0} & 0 \end{pmatrix}X^{k_1}$.

(2)  For $Y=\begin{pmatrix} \alpha_Y & \mathbf{0}\\ \mathbf{c}_Y & \delta_Y
\end{pmatrix}$ let us compute, $A- \begin{pmatrix} \beta_1 & (1,0,0)\\
(0,1,0) & \beta_8 \end{pmatrix}Y^{k_2}=$
$$\begin{pmatrix} \alpha - \beta_1\alpha_Y^{k_2}-\langle(1,0,0), \mathbf{c}_Y\rangle & \mathbf{b}- \delta_Y^{k_2} (1,0,0) -(0,1,0) \wedge f(\alpha_Y,\delta_Y,k_2)\mathbf{b}_Y \\ \mathbf{c} -\alpha_Y^{k_2} (0,1,0) - \beta_8 f(\alpha_Y,\delta_Y, k_2) \mathbf{c}_Y & \delta - \beta_8\delta_Y^{k_2}
\end{pmatrix}.
$$
Let $\delta_Y$ be such that $\delta = \beta_8 \delta_Y^{k_2}$. Choose $\alpha_Y$ such that $f(\alpha_Y, \delta_y, k_2) = \beta_8^{-1}$. Then, $\mathbf{c}_Y = c - \alpha_Y^{k_2}(0,1,0)$. Using \cref{lem: diag-2}, there exist $X\in \Oc(\F)$ such that
$$ A- \begin{pmatrix} \beta_1 & (1,0,0)\\ (0,1,0) & \beta_8 \end{pmatrix}Y^{k_2} = \begin{pmatrix} \alpha' & \mathbf{b}'\\
\mathbf{0} & 0 \end{pmatrix} =  \begin{pmatrix} \alpha_1 & \mathbf{0}\\ \mathbf{0} & 0 \end{pmatrix}X^{k_1}.$$
\end{proof}

\begin{proposition}
    \label{prop:diag2-beta5}
    For $\alpha_8,\beta_5, \beta_1 \in\F^\times$ and $A\in\Oc(\F)$, there  exist $X$ and $Y$ in $\Oc(\F)$ such that 
\begin{enumerate}
\item $$A= \begin{pmatrix}0 & \mathbf{0}\\ \mathbf{0} & \alpha_8 \end{pmatrix}X^{k_1} + \begin{pmatrix} \beta_1 & (1,0,0) \\ (\beta_5,0,0) & \beta_8 \end{pmatrix}Y^{k_2}.$$
\item \begin{equation}\label{eqn:diag2-3}
    A=\begin{pmatrix} 0 & \mathbf{0}\\ \mathbf{0} & \alpha_8    \end{pmatrix}X^{k_1} + \begin{pmatrix} \beta_1 & (1,0,0)\\    (0,1,0) & \beta_8 \end{pmatrix}Y^{k_2}.
\end{equation}
\end{enumerate}
\end{proposition}
\begin{proof}
    The proof follows a similar approach to the previous proposition.
\end{proof}

\begin{corollary}
\label{coro:diag-1}
Let $\beta_1\in\F$. When $A_1= \begin{pmatrix} \alpha_1 & \mathbf{0}\\ \mathbf{0} & 0 \end{pmatrix}$ and $A_2 = \begin{pmatrix} \beta_1 & (1,0,0)\\
\mathbf{0} & 0 \end{pmatrix}$ or $\begin{pmatrix} \beta_1 & (1,0,0)\\ (0,1,0) & 0 \end{pmatrix}$ the map $A_1(X^{k_1}) + A_2(Y^{k_2})$ is not surjective.
\end{corollary}
\begin{proof}
From Proposition~\ref{prop:diag-upper} we note that if $\beta_8 = 0$ in \cref{eqn: diag-1}, for any choice of $X,Y\in\Oc(\F)$ we have  $\begin{pmatrix} \alpha_1 & \mathbf{0}\\ \mathbf{0} & 0
\end{pmatrix}X^{k_1} + \begin{pmatrix} \beta_1 & (1,0,0)\\  \mathbf{0} & 0 \end{pmatrix}Y^{k_2} = \begin{pmatrix} \alpha & \mathbf{b}\\ \mathbf{c} & 0 \end{pmatrix}$. Hence, in this case, the map is not surjective.

We note that in the Proposition~\ref{prop:diag-beta5} above (see \cref{eqn:diag-3}) if $\beta_8=0$, for any $X,Y\in\Oc(\F)$ we have
$ \begin{pmatrix} \alpha_1 & \mathbf{0}\\ \mathbf{0} & 0 \end{pmatrix}X^{k_1} + \begin{pmatrix} \beta_1 & (1,0,0)\\
(0,1,0) & 0 \end{pmatrix}Y^{k_2} = \begin{pmatrix} \alpha' & \mathbf{b}'\\ (0,c_2',c_3') & \delta' \end{pmatrix},$ 
which shows that the map is not surjective in this case.
\end{proof}

\begin{corollary}
\label{coro:diag-2}
When $A_1= \begin{pmatrix} 0 & \mathbf{0}\\ \mathbf{0} & \alpha_8 \end{pmatrix}$ and $A_2 = \begin{pmatrix} 0 & (1,0,0)\\
\mathbf{0} & \beta_8 \end{pmatrix}$ or $\begin{pmatrix} 0 & (1,0,0)\\ (0,1,0) & \beta_8 \end{pmatrix}$ the map $A_1(X^{k_1}) + A_2(Y^{k_2})$ is not surjective.
\end{corollary}
\begin{proof}
    From \cref{prop:diag2-upper} we note that if $\beta_1 = 0$ in \cref{eqn: diag-2}, for any choice of $X,Y\in\Oc(\F)$ we have  $\begin{pmatrix} 0 & \mathbf{0}\\ \mathbf{0} & \alpha_8
\end{pmatrix}X^{k_1} + \begin{pmatrix} 0 & (1,0,0)\\  \mathbf{0} & \beta_8 \end{pmatrix}Y^{k_2} = \begin{pmatrix} \alpha & (b_1,0,0)\\ \mathbf{c} & \delta \end{pmatrix}$. Hence, in this case, the map is not surjective.
We note that in the \cref{prop:diag2-beta5} above (see \cref{eqn:diag2-3}) if $\beta_1=0$, for any $X,Y\in\Oc(\F)$ we have
$ \begin{pmatrix} 0 & \mathbf{0}\\ \mathbf{0} & \alpha_8 \end{pmatrix}X^{k_1} + \begin{pmatrix} 0 & (1,0,0)\\
(0,1,0) & \beta_8 \end{pmatrix}Y^{k_2} = \begin{pmatrix} \alpha & (b_1,0,b_3)\\ \mathbf{c} & \delta\end{pmatrix},$ 
which shows that the map is not surjective in this case.
\end{proof}
\subsection{The Remaining non-unit cases}
Now, we deal with all the remaining cases. We will represent $A = \begin{pmatrix} \alpha & (b_1, b_2, b_1) \\ (c_1, c_2, c_3) & \delta \end{pmatrix} \in \Oc(\F)$ and look for $X, Y$ such that $A= A_1(X^{k_1}) + A_2 (Y^{k_2})$. We begin with the following,

\begin{lemma}\label{lem:nilpotent}
Let $\begin{pmatrix} \alpha & (b_1,0,0) \\ (0,c_2, c_3) & 0
\end{pmatrix}\in \Oc(\F)$. For a positive integer $k_1$, there exist $X\in \Oc(\F)$ such that 
$\begin{pmatrix} \alpha & (b_1,0,0) \\ (0,c_2,c_3) & 0  \end{pmatrix} = \begin{pmatrix} 0 & (1,0,0)\\ \mathbf{0} & 0 \end{pmatrix}X^{k_1}$.  
\end{lemma}
\begin{proof} For $X = \begin{pmatrix} \alpha_X & (0,c_3,c_2)\\    (\alpha, 0, 0) & \delta_X \end{pmatrix}$ we get, 
\begin{align*}
\begin{pmatrix} 0 & (1,0,0)\\ \mathbf{0} & 0 \end{pmatrix}X^{k_1} = \begin{pmatrix} \langle f(\alpha_X,\delta_X,k_1)(\alpha,0,0),(1,0,0)\rangle & \delta_X^{k_2}(1,0,0)\\ (0,c_2,c_3) & 0
\end{pmatrix}.  
\end{align*}
Choose $b_1$ such that $b_1 = \delta_X^{k_2}$. Let $\alpha_X$ be such that $f(\alpha_X, \delta_X, k_1) = 1$. This completes the proof.
\end{proof}

Note that the equation $ A-\begin{pmatrix} 0 & (1,0,0)\\   \mathbf{0} & 0 \end{pmatrix}X^{k_1} = 0$ has a solution if and only if $A = \begin{pmatrix} \alpha & (b_1,0,0)\\ (0,c_2,c_3) & 0  \end{pmatrix}\in\Oc(\F).$

\begin{proposition}\label{prop:nilpotent-lowercase-1}
Let $A\in\Oc(\F)$, then there exist $X,Y\in\Oc(\F)$ such that 
\begin{enumerate}
\item for $\beta_5\in\F^\times$ we have  \begin{align*}\label{eqn:nilpotent-1}
A=\begin{pmatrix} 0 & (1,0,0)\\ \mathbf{0} & 0 \end{pmatrix}X^{k_1} + \begin{pmatrix} \beta_1 & \mathbf{0}\\ (\beta_5,0,0) & \beta_8    \end{pmatrix}Y^{k_2}. \end{align*}
\item $A = \begin{pmatrix} 0 & (1,0,0)\\ \mathbf{0} & 0 \end{pmatrix}X^{k_1} + \begin{pmatrix} 0 & \mathbf{0}\\ (0,1,0) & 0
\end{pmatrix}Y^{k_2} $ if and only if $A= \begin{pmatrix} \alpha & (b_1,0,b_3)\\ (0,c_2,c_3) & \delta \end{pmatrix}$.
\end{enumerate}
\end{proposition}
\begin{proof} 
For the proof of (1), let $Y=\begin{pmatrix} \alpha_Y & \mathbf{b}_Y \\ \mathbf{c}_Y & \delta_Y \end{pmatrix}$ where $\alpha_Y^{k_2} = \beta_5^{-1}c_1$, $\mathbf{c}_Y = \beta_5^{-1}(0,b_3,b_2)$, $\delta_Y$ be such that $f(\alpha_Y,\beta_Y, k_2) = 1$ and $\mathbf{b}_Y=\beta_5^{-1}(\delta-\beta_8\delta_Y^{k_2},0,0)$.  Then, $A-\begin{pmatrix} \beta_1 & \mathbf{0}\\ (\beta_5,0,0) & \beta_8 \end{pmatrix}Y^{k_2} = \begin{pmatrix} \alpha' & (b_1',0,0)\\ (0,c_2',c_3') & 0      \end{pmatrix}$. By \cref{lem:nilpotent}, there exists $X\in\Oc(\F)$ for which the solution of the equation exists.

For the proof of (2), let $Y=\begin{pmatrix} 0 & (0,\delta,0) \\   (-b_3,0,0) & 1 \end{pmatrix}$. Then by \cref{lem:nilpotent}, there exist $X\in \Oc(\F)$ such that 
\begin{eqnarray*} A -\begin{pmatrix}   0 & \mathbf{0}\\ (0,1,0) & 0 \end{pmatrix}Y^{k_2} &= \begin{pmatrix}
\alpha & (b_1,0,b_3)-(0,1,0)\wedge (-b_3,0,0)\\ (0,c_2,c_3) & \delta-\langle(0,1,0),(0,\delta,0)\rangle \end{pmatrix}\\ &=\begin{pmatrix} \alpha & (b_1,0,0)\\ (0,c_2,c_3) & 0 \end{pmatrix}=  \begin{pmatrix} 0 & (1,0,0)\\ \mathbf{0} & 0     \end{pmatrix}X^{k_1}. 
\end{eqnarray*}
\end{proof}

\begin{proposition}
\label{prop:nilpotent-upper-dist-1}
Let $A\in\Oc(\F)$, then there exist $X,Y\in \Oc(\F)$ such that
\begin{enumerate}
\item for $\beta_8\in\F^\times$, the equation $A= \begin{pmatrix} 0 & (1,0,0)\\ \mathbf{0} & 0     \end{pmatrix}X^{k_1}+\begin{pmatrix} 0 & (0,1,0)\\ \mathbf{0} & \beta_8 \end{pmatrix}Y^{k_2}$ holds if and only if $A=\begin{pmatrix} \alpha & (b_1,\beta_8^{-1}\delta,0)\\ \mathbf{c} & \delta \end{pmatrix}$.
\item For $\beta_1\in\F^\times$, the equation $A=\begin{pmatrix} 0 & (1,0,0)\\ \mathbf{0} & 0 \end{pmatrix}X^{k_1} + \begin{pmatrix} \beta_1 & (0,1,0)\\ \mathbf{0} & 0 \end{pmatrix}Y^{k_2}$ holds if and only if $A=\begin{pmatrix} \alpha & (b_1,b_2,\beta_1\tau)\\
    (\tau,c_2,c_3) & 0 \end{pmatrix}$.
\item $A=  \begin{pmatrix} 0 & (1,0,0)\\ \mathbf{0} & 0  \end{pmatrix} X^{k_1} + \begin{pmatrix} 0 & (0,1,0)\\ \mathbf{0} & 0
\end{pmatrix}Y^{k_2}$ if and only if $A = \begin{pmatrix} \alpha & (b_1,b_2,0)\\ \mathbf{c} & 0 \end{pmatrix}$.
\end{enumerate}
\end{proposition}
\begin{proof}
For the proof of (1), when $A=\begin{pmatrix} \alpha & (b_1,\beta_8^{-1}\delta,0)\\   \mathbf{c} & \delta \end{pmatrix}$, let $Y=\begin{pmatrix}   \alpha_Y & \mathbf{0}\\ \beta_8^{-1}\mathbf{c} & \delta_Y \end{pmatrix}$. Then, 
\begin{align*}
A - \begin{pmatrix} 0 & (0,1,0)\\ \mathbf{0} & \beta_8 \end{pmatrix}Y^{k_2} & =\begin{pmatrix} \alpha-f(\alpha_Y,\delta_Y,k_2)\langle\mathbf{c},(0,1,0)\rangle & (b_1,\delta-\beta_8^{-1}\delta_Y^{k_2},0) \\  \mathbf{c}-f(\alpha_Y,\delta_Y,k_2)\mathbf{c} & \delta-\beta_8\delta_Y^{k_2}
\end{pmatrix}
\end{align*}
This equals $\begin{pmatrix} \alpha' & (b_1,0,0)\\ \mathbf{0} & 0
\end{pmatrix}$, if $\delta_Y^{k_2} = \beta_8^{-1}\delta$ and $\alpha_Y$ be such that $f(\alpha_Y, \delta_Y, k_2) = 1$. The existence of $X$ follows from \cref{lem:nilpotent}. A straightforward calculation shows that the converse is also true.

Now we prove (2). For $A=\begin{pmatrix} \alpha & (b_1,b_2,\beta_1\tau)\\(\tau,c_2,c_3) & 0 \end{pmatrix}$, let $Y=\begin{pmatrix} 1 & (0, \beta_1^{-1}b_2, \tau)\\ \mathbf{0} & 0 \end{pmatrix}$. Then, by \cref{lem:nilpotent},
\begin{align*}
A-\begin{pmatrix} \beta_1 & (0,1,0)\\ \mathbf{0} & 0 \end{pmatrix}Y^{k_2} = \begin{pmatrix} \alpha' & (b_1,0,0)\\ (0,c_2,c_3) & 0 \end{pmatrix} = \begin{pmatrix} 0 & (1,0,0)\\
\mathbf{0} & 0 \end{pmatrix}X^{k_1}.
\end{align*}
The other way follows easily.

Let us quickly see proof of (3). Let $Y = \begin{pmatrix} \alpha_Y & (0,0,c_1)\\ \mathbf{0} & \delta_Y \end{pmatrix}$. Choose $\delta_Y$ such that $b_2=\delta_Y^{k_2}$ and $\alpha_Y$ so that $f(\alpha_Y,\delta_Y,k_2)=1$. Then by \cref{lem:nilpotent}
$A-\begin{pmatrix} 0 & (0,1,0)\\ \mathbf{0} & 0 \end{pmatrix}Y^{k_2} = \begin{pmatrix} \alpha & (b_1,0,0)\\ (0,c_2,c_3) & 0 \end{pmatrix} =\begin{pmatrix} 0 & (1,0,0)\\   \mathbf{0} & 0 \end{pmatrix}X^{k_1}$ for some $X\in\Oc(\F)$.
\end{proof}

\begin{proof}[Proof of \cref{mainthm}] The \cref{pro:reprsen-simul} gives the representatives of the orbits of the $G_2$-action on $\Oc(\F)$. Using \cref{coro:invertible}, if either $A_1$ or $A_2$ is invertible, then the map is surjective.
This covers the cases of $\mathrm{(EK_1)}$, $\mathrm{(K_1E)}$ fully.
In the following, we summarize the cases where neither $A_1$ nor $A_2$ is invertible.

The case of $\mathrm{(DD)}$ follows from \cref{prop:diag}.
When the orbits are of type $\mathrm{(FK)}$ the result follows from \cref{prop:diag-upper} and \cref{prop:diag2-upper}.
 The cases of $\mathrm{(FN)}$ and $\mathrm{(FP)}$ are treated in \cref{prop:diag-beta5} and \cref{prop:diag2-beta5}.
 When the orbit type is $\mathrm{(K_1L^T)}$ has been analyzed in \cref{prop:nilpotent-lowercase-1}.

 Finally, the result now follows from the results about non-surjectivity which are discussed in \cref{prop:diag}, \cref{coro:diag-1}, \cref{coro:diag-2}, \cref{lem:nilpotent}, \cref{prop:nilpotent-lowercase-1} and \cref{prop:nilpotent-upper-dist-1}.
\end{proof}

\bibliographystyle{plain} 
\bibliography{ref}

\end{document}